\documentclass[12pt,reqno,dvips]{amsart}
\usepackage{amsmath,amsfonts,amssymb,amsthm,enumerate,fleqn}
\def\CC{\mathbb{C}}

\def\RR{\mathbb{R}}
\def\NN{\mathbb{N}}

\theoremstyle{plain}
\newtheorem{theorem}{\bf Theorem}[section]
\newtheorem{lemma}[theorem]{\bf Lemma}
\theoremstyle{remark}
\newtheorem{definition}[theorem]{\bf Definition}

\newtheorem*{observation}{\bf Observations}
\newtheorem{example}[theorem]{\bf Example}

\newtheorem{remark}[theorem]{\bf Remark}

\title[Approximative $K$-Atomic Decompositions and frames in Banach Spaces]{Approximative $K$-Atomic Decompositions and frames in Banach Spaces}

\author[S. Jahan]{Shah Jahan}
\address{Shah Jahan, Department of Mathematics,
University of Delhi, Delhi-110007, India}
\email{chowdharyshahjahan@gmail.com}

\begin{document}
\begin{abstract}\baselineskip12pt
[L. Gavruta, Frames for Operators, Appl. comput. Harmon. Anal. 32(2012), 139-144] introduced a special kind of frames, named $K$-frames, where $K$ is an operator, in Hilbert spaces, is significant in frame theory and has many applications. In this paper, first of all, we have introduced the notion of approximative $K$-atomic decomposition in Banach spaces.  We gave two characterizations regarding the existence of approximative $K$-atomic decompositions in Banach spaces. Also some results on the existence of approximative $K$-atomic decompositions are obtained. We discuss several methods to construct approximative $K$-atomic decomposition for Banach Spaces. Further, approximative $\mathcal{X}_d$-frame and approximative $\mathcal{X}_d$-Bessel sequence are introduced and studied. Two necessary conditions are given under which an approximative $\mathcal{X}_d$-Bessel sequence and approximative $\mathcal{X}_d$-frame give rise to a bounded operator with respect to which there is an approximative $K$-atomic decomposition. Examples and counter examples are provided to support our concept. Finally, a possible application is given.
\end{abstract}
\thispagestyle{empty}
\maketitle
\thispagestyle{empty}

\baselineskip15pt
\section{Introduction and Preliminaries}
Fourier transform has been a major tool in analysis for over a century. It has a serious lacking for signal analysis in which it hides in its phase information concerning the moment of emission and duration of a signal. What actually needed was a localized time frequency representation which has this information encoded in it. In 1946, Dennis Gabor \cite{6} filled this gap and formulated a fundamental approach to signal decomposition in terms of elementary signals. On the basis of this development, in 1952, Duffin and Schaeffer \cite{4} introduced frames for Hilbert spaces to study some deep problems in non-harmonic Fourier series. In fact, they abstracted the fundamental notion of Gabor for studying signal processing.
 Let $\mathcal{H}$ be a real (or complex) separable Hilbert space with inner product $\langle .,. \rangle.$ A countable sequence $\{f_n\} \subset \mathcal{H}$ is called a frame for the Hilbert space $\mathcal{H},$ if there exist positive constants $A,~ B >0$ such that
\begin{align}\label{1} A\| f\|^2_\mathcal{H} \leq \sum_{n=1}^\infty |\langle f, f_n \rangle|^2 \leq B\|f\|^2_\mathcal{H},~~~ \mbox{for all}~~ \in \mathcal{H}
 \end{align} The positive numbers $A$ and $B$ are called the lower and upper frame bounds of the frame, respectively. These bounds are not unique. The inequality in (\ref{1}) is called the frame inequality of the frame. If $\{f_n\}$ is a frame for $\mathcal{H}$ then the following operators are associated with it.
  \begin{enumerate}[(a)]
 \item Pre-frame operator $T: l^2(\mathbb{N}) \longrightarrow \mathcal{H}$ is defined as $T\{c_n\}_{n=1}^\infty= \sum\limits_{k=1}^\infty c_nf_n,~~\{c_n\}_{n=1}^\infty \in l^2(\mathbb{N}).$
 \item Analysis operator $T^*: \mathcal{H} \longrightarrow l^2(\mathbb{N}), T^*= \{\langle f, f_k\rangle\}_{k=1}^\infty~~~ f \in \mathcal{H}.$
 \item Frame operator $S=TT^*=: \mathcal{H} \longrightarrow \mathcal{H},~~ S_f= \sum\limits_{k=1}^\infty \langle f, f_k \rangle f_k, ~~f \in \mathcal{H}.$ The frame $S$ is bounded, linear and invertible on $\mathcal{H}$. Thus, a frame for $\mathcal{H}$ allows each vector in $\mathcal{H}$ to be written as a linear combination of the elements in the frame, but the linear independence between the elements is not required; i.e for each vector $f \in \mathcal{H}$ we have,
     \begin{eqnarray*}
     f=SS^{-1}f= \sum\limits_{k=1}^\infty \langle f, f_k \rangle f_k.
     \end{eqnarray*}
   \end{enumerate}
   For more details related to frame and Riesz bases in Hilbert spaces, one may refer to \cite{a,3}. These ideas did not generate much interest outside of non-harmonic Fourier series and signal processing for more than three decades until Daubechies et al. \cite{g} reintroduced frames. After this landmark paper the theory of frames begin to be studied widely and found many applications to wavelet and Gabor transforms in which frames played an important role. Feichtinger and Gr\"{o}cheing \cite{5} extended the idea of Hilbert frames to Banach spaces and called it atomic decomposition. A more general concept called Banach frame was introduced by Gr\"{o}cheing \cite{8} and were further studied in \cite{SKK,WZ}. Banach frames were developed for the theory of frames in the context of Gabor and Wavelet analysis. Christensen and Heil \cite{2} studied some perturbation results for Banach frames and atomic decompositions.\\
    In particular, frames are widely used in sampling theory in \cite{ag} amounts to the construction of Banach frames consisting of reproducing kernals for a large class of shift invariant spaces.  Aldroubi et al. \cite{ab} used Banach frames in various irregular sampling problems. Elder and Forney \cite{qc} used tight frames for quantum measurement. Gr\"{o}chenig \cite{gl} emphasised that localization of a frame is a necessary condition for its extension to a Banach frame for the associated Banach spaces. He also observed that localized frames are universal Banach frames for the associated family of Banach spaces.
Fornasier \cite{fo}  studied Banach frames for $\alpha$-modulation spaces. In fact, he gave a Banach frame characterization for the $\alpha$-modulation spaces. Shah et.al  \cite{SVC} defined and studied Banach frames to a new geometric notation; in fact they gave a sufficient condition and a necessary condition for a cone associated with a Banach frame to be a generating cone.\\
 Casazza et al. \cite{CCS} studied $\mathcal{X}_d$-frames and $\mathcal{X}_d$-Bessel sequences in Banach spaces. Stoeva \cite{DT} gave some perturbation results for $\mathcal{X}_d$-frames and atomic decompositions. Kaushik and Sharma \cite{n} studied approximative atomic decompositions in Banach spaces. For further studies related to approximative frame one may refer \cite{SH, SS,IS1}. Gavruta \cite{7} introduced and studied atomic system for an operator $K$ and the notion of $K$-frame in a Hilbert space. Xiao et al. \cite{12} discussed relationship between $K$-frames and ordinary frames in Hilbert spaces. Poumai and Jahan \cite{KA} introduced K-atomic decompositions in Banach spaces.\\
 \textbf{outline of the paper.}
 In this paper, we have introduced the notion of approximative $K$-atomic decomposition in Banach spaces.  We gave two characterizations regarding the existence of approximative $K$-atomic decompositions in Banach spaces. Also some results on the existence of approximative $K$-atomic decompositions are obtained. We discuss several methods to construct approximative $K$-atomic decomposition for Banach Spaces. Further, approximative $\mathcal{X}_d$-frame and approximative $\mathcal{X}_d$-Bessel sequence are introduced and studied. Two necessary conditions are given under which an approximative $\mathcal{X}_d$-Bessel sequence and approximative $\mathcal{X}_d$-frame give rise to bounded operators with respect to which there is an approximative $K$-atomic decomposition. Examples and counter examples are provided to support our concept of approximative K-atomic decomposition. Finally, we gave a possible application of our work.\\
Next we give some basic notations. Throughout this paper, $\mathcal{X}$ will denote a separable Banach space over the scalar field K($\RR$ or $\CC$), $\mathcal{X}^*$ the dual space of $\mathcal{X}$, $\mathcal{X}_d$ a BK-space and $L(\mathcal{X},\mathcal{Y} )$ will denote the space of all bounded linear operators from $\mathcal{X}$ into $\mathcal{Y}$. For $ T\in L(\mathcal{X}) $, $T^{*}$ denotes the adjoint of $T$, $\pi:\mathcal{X} \longrightarrow \mathcal{X}^{**}$ is the natural canonical projection from $\mathcal{X}$ onto $\mathcal{X}^{**}.$ Also $T^\dagger$  denote the pseudo inverse of the operator $T.$ Note that $TT^\dagger f=f$ for all $ f \in R(K).$ Throughout $R(K)$ is closed.
\begin{definition}
 \cite{8}Let $\mathcal{X}$ be a Banach space and $\mathcal{X}_d$ be a BK-space. A sequence $(x_n,f_n)(\lbrace x_n \rbrace \subset \mathcal{X},\lbrace f_n \rbrace \subset \mathcal{X}^*)$ is called an \emph{atomic decomposition} for $\mathcal{X}$ with respect to $\mathcal{X}_d$ if the following statements hold:
\begin{enumerate}[(a)]
\item $\lbrace f_n(x)\rbrace\in \mathcal{X}_d$, for all $x\in ~\mathcal{X}$.
\item There exist constants $A$ and $B$ with $0<A\leq B<\infty$ such that
\begin{eqnarray}
A\Vert x\Vert_\mathcal{X}~\leq~\Vert\lbrace f_n(x)\rbrace\Vert_{\mathcal{X}_d}~\leq~ B\Vert x\Vert_\mathcal{X},~ \text{for all} \ x\in \mathcal{X}
\end{eqnarray}

\item $x= \sum\limits_{n=1}^{\infty} f_n(x)x_n$,\ \ for all $x\in \mathcal{X}$.
\end{enumerate}
\end{definition}
Next, we state some lemmas which we will use in the subsequent results.
\begin{lemma}\label{L1}
 \cite{TL,WZ}Let $\mathcal{X}$, $\mathcal{Y}$ be Banach spaces and $T:\mathcal{X} \longrightarrow \mathcal{Y}$ be a bounded linear operator. Then, the following conditions are equivalent:
\begin{enumerate}[(a)]
\item There exist two continuous projection operators $P:\mathcal{X} \rightarrow \mathcal{X} $ and $Q:\mathcal{Y} \rightarrow \mathcal{Y}$ such that
\begin{eqnarray}
P(\mathcal{X})=kerT ~and~ Q(\mathcal{Y})=T(\mathcal{X}).
\end{eqnarray}
\item $T$ has a pseudo inverse operator $T^\dagger$.
\end{enumerate}
If two continuous projection operators $P:\mathcal{X}\rightarrow \mathcal{X}$ and $Q:\mathcal{Y} \rightarrow \mathcal{Y}$ satisfies (2.3), then there exists a pseudo inverse operator $T^\dagger$ of $T$ such that $T^\dagger T=I_\mathcal{X}-P$ \ and \ $TT^\dagger =Q$, where $I_\mathcal{X}$ is the identity operator on $\mathcal{X}$.
\end{lemma}
\begin{lemma}
 \cite{1, 10}Let $\mathcal{X}$ be a Banach space. If $T\in L(\mathcal{X})$ has a generalized inverse S$\in L(\mathcal{X})$, then $TS$, $ST$ are projections and $TS(\mathcal{X})=T(\mathcal{X})$ and $ST(\mathcal{X})=S(\mathcal{X})$.
\end{lemma}
\begin{lemma}\label{KL}
\cite{n, IS2} Let $\mathcal{X}$ be a Banach space and $\lbrace f_n \rbrace \subset \mathcal{X}^*$ be a sequence such that $ \lbrace x \in \mathcal{X} : f_n(x) = 0, for \ all  \ n \in \NN \rbrace = \lbrace 0 \rbrace$. Then $\mathcal{X}$ is linearly isometric to the Banach space $\mathcal{X}_d= \lbrace \lbrace f_n(x) \rbrace : x \in \mathcal{X} \rbrace$, where the norm is given by $ \Vert \lbrace f_n(x) \rbrace \Vert_{\mathcal{X}_d}$=$\Vert x \Vert_\mathcal{X}$, $x \in \mathcal{X}$.
\end{lemma}
\section{Main Results}
  Poumai and Jahan \cite{KA} defined and studied $K$-atomic decomposition as a generalization of $K$-frames in Banach spaces. Here we shall extend this study further and introduced the concept of approximative K-atomic decomposition in Banach spaces and obtained new and interesting results. We start this section with the following definition of approximative K-Atomic decomposition:
\begin{definition}
Let $\mathcal{X}$ be a Banach Space, $\lbrace x_n\rbrace\subset \mathcal{X},\{h_{n,i}\}\underset{n \in \mathbb{N}}{_{i=1,2,3,...,m_n}}\subset \mathcal{X}^*,$ where $\{m_n\}$ is an increasing sequence of positive integer and $K \in L(\mathcal{X})$. A pair $(\{x_n\},\{h_{n,i}\}\underset{n \in \mathbb{N}}{_{i=1,2,3,...,m_n}})$ is called an \emph{approximative K-atomic decomposition} for $\mathcal{X}$ with respect to $\mathcal{X}_d,$ if the following statements holds:
\begin{enumerate}[(a)]
\item $\{h_{n,i}(x)\}\underset{n \in \mathbb{N}}{_{i=1,2,3,...,m_n}} \in \mathcal{X}_d$, for all $x\in \mathcal{X}.$
\item There exist constants $A$ and $B$ with $0< A\leq B < \infty$ such that
\begin{eqnarray*}
A\parallel K(x)\parallel_\mathcal{X}~\leq~ \parallel \{h_{n,i}(x)\}\underset{n \in \mathbb{N}}{_{i=1,2,3,...,m_n}}\Vert_{\mathcal{X}_d} \leq B\parallel x\parallel_\mathcal{X},~ \mbox{for all} \ x\in \mathcal{X}.
\end{eqnarray*}
\item $\lim\limits_{n \rightarrow \infty}\sum\limits_{i=1}^{m_n}h_{n,i}(x)x_i$ converges for all $x\in \mathcal{X}$ and\\ $K(x)=\lim\limits_{n \rightarrow \infty}\sum\limits_{i=1}^{m_n}h_{n,i}(x)x_i$.
\end{enumerate}

 The constants $A$ and $B$ are called lower and upper bounds of the approximative $K$-atomic decomposition $(\{x_n\},\{h_{n,i}\}_{i=1,2,3,...,m_n}).$
\end{definition}
\begin{observation}
 If $(\{x_n\},\{f_n\})$ is a $K$-atomic decomposition for $\mathcal{X}$ with respect to $\mathcal{X}_d,$ then for $h_{n,i}=f_i, i=1,2,...,n, n\in \mathbb{N},$ $(\{x_n\},\{h_{n,i}\})$ is an approximative $K$-atomic decomposition for $\mathcal{X}$ with respect to some associated Banach space $\mathcal{X}_d.$
\end{observation}
\begin{remark}
Let $(x_n,h_{n,i})$ be an approximative $K$-atomic decomposition for $\mathcal{X}$ with respect to $\mathcal{X}_d$ with bounds $A$ and $B$.
\subsubsection*{(I)}
If $K=I_\mathcal{X}$, then $(x_n,h_{n,i})$ is an approximative atomic decomposition for $\mathcal{X}$ with respect to $\mathcal{X}_d$  with bounds $A$ and $B$.
\subsubsection*{(II)}
If $K$ is invertible, then $(K^{-1}(x_n),h_{n,i})$ is an approximative atomic decomposition for $\mathcal{X}$ with respect to $\mathcal{X}_d$.
\end{remark}
In the following example, we show the existence of approximative $K$-atomic decomposition for a Banach space $\mathcal{X}$ with respect to an associated BK space $\mathcal{X}_d$ .
\begin{example}
Let $\mathcal{X}$ be a Banach Space. Let$ \lbrace x_n\rbrace \subseteq \mathcal{X}$, $\lbrace h_{n,i} \rbrace \subseteq \mathcal{X}^*$ such that $\lim\limits_{n \to \infty} \sum\limits_{i=1}^{m_n}h_{n,i}(x)x_n$ converges for all $x\in \mathcal{X}$ and $x_n\neq 0,$ for all $n\in \mathbb{N}$. Also, let $\mathcal{X}_d=\lbrace\lbrace h_{n,i}\rbrace\vert \lim\limits_{n \to \infty}\sum\limits_{i=1}^{m_n}h_{n,i}x_i ~\mbox{converges}\rbrace.$ Then $\mathcal{X}_d$ is a BK-space with norm $\Vert \lbrace h_{n,i}(x)\rbrace\Vert_{\mathcal{X}_d}=\sup\limits_{1\leq n< \infty}\parallel \sum\limits_{i=1}^{n}h_{n,i} x_i\parallel$. Define an operator as $T:\mathcal{X}_d\longrightarrow \mathcal{X}$ as $T\lbrace h_{n,i}\rbrace=\lim\limits_{n \to \infty}\sum\limits_{i=1}^{m_n}h_{n,i} x_i$ and define $S:\mathcal{X}\longrightarrow \mathcal{X}_d$ as $S(x)=\lbrace h_{n,i}(x)\rbrace, ~x \in \mathcal{X}$. Take $K=TS.$ Then $K :\mathcal{X}\longrightarrow \mathcal{X}$ is such that $K(x)=TS(x)=\lim\limits_{n \to \infty}\sum\limits_{i=1}^{m_n}h_{n,i}(x)x_i,$ for all $x\in \mathcal{X},~ i=1,2,...,n,~ n \in \mathbb{N}.$ Clearly, $\lbrace h_{n,i}(x)\rbrace \in \mathcal{X}_d$ and
\begin{eqnarray*}
\Vert K(x)\Vert_\mathcal{X} &=& \lim\limits_{n \to \infty}\bigg\|\sum\limits_{i=1}^{m_n}h_{n,i}(x) x_i\bigg\Vert\leq \sup\limits_{1\leq n< \infty}\bigg\| \sum\limits_{k=1}^{n}h_k(x) x_k\bigg\| \\
&=&\Vert\lbrace h_{n,i}(x)\rbrace\Vert_{\mathcal{X}_d}\leq C \parallel x\parallel_\mathcal{X},~\mbox{for all}~ x\in \mathcal{X},
\end{eqnarray*} where $C=\sup\limits_{1\leq n < \infty}\parallel S_n \parallel$ and $S_n(x)=\lim\limits_{n \to \infty}\sum\limits_{i=1}^{m_n} h_{n,i}(x) x_i.$
\\Hence, $(x_n,h_{n,i})$ is an approximative $K$-atomic decomposition for $\mathcal{X}$ with respect to $\mathcal{X}_d.$
\end{example}
In the following result, we give the characterization regarding the existence of approximative $K$-atomic decompositions in Banach spaces.
\begin{theorem}\label{T1} Let $K \in L(\mathcal{X})$ with $K\neq 0.$ Then a Banach space $\mathcal{X}$ has an approximative $K$-atomic decomposition if and only if there exists a sequence $\{v_i\} \subset B(\mathcal{X})$ of finite rank endomorphism such that $K(x)= \sum\limits_{i=1}^nv_i(x), ~~x \in \mathcal{X}.$
\end{theorem}
\begin{proof} Let $\{x_n\} \subset \mathcal{X}$ and $\{h_{n,i}\}\underset{n \in \mathbb{N}}{_{i=1,2,3,...,m_n}}\subset \mathcal{X}^*,$ where $\{m_n\}$ is an increasing sequence of positive integer such that $(\{x_n\},\{h_{n,i}\}\underset{n \in \mathbb{N}}{_{i=1,2,3,...,m_n}})$ is an \emph{approximative K-atomic decomposition} for $\mathcal{X}$ with respect to $\mathcal{X}_d.$ Define
\begin{align*} S_n(x)=\sum_{i=1}^{m_n} h_{n,i}(x)x_i, ~~\mbox{for all}~ x \in \mathcal{X}, n \in \mathbb{N}.
\end{align*} Then for each $n \in \mathbb{N}$ and $x \in \mathcal{X}$, $S_n(x)$ is a well defined continuous linear mapping on $\mathcal{X}$ such that $\lim\limits_{n \to \infty}S_n(x)=x, ~~x \in \mathcal{X}.$ Also by uniform boundedness principle we have $\sup\limits_{1\leq n \leq \infty}\|S_n(x)\|< \infty.$
Assume that $v_1=S_1$,~~$v_{2n}=v_{2n+1}=\frac{1}{2}(S_{n+1}-S_n)$,$n \in \mathbb{N}.$ Now, we compute
\begin{align*} \lim\limits_{n \to \infty}\sum_{i=1}^nv_i(x)&= \lim\limits_{n \to \infty}(S_1(x)+\frac{1}{2}(S_2(x)-S_1(x))+\frac{1}{2}(S_2(x)-S_1(x))\\&+\frac{1}{2}(S_3(x)-S_2(x))+\frac{1}{2}(S_3(x)-S_2(x))+...)
\\&=\lim\limits_{n \to \infty}S_n(x)
\\&=K(x), ~~\mbox{for all}~~x \in \mathcal{X}, K \in L(X).
\end{align*} Therefore, $\lim\limits_{n \to \infty}\sum_{i=1}^nv_i(x) =K(\mathcal{X}).$\\
Conversely assume that there exists a sequence of finite rank endomorphism $\{S_n\} \subset L(\mathcal{X})$ such that $\lim\limits_{n \to \infty}S_n(x)=K(x), ~~x \in \mathcal{X}.$ Then, each $S_n(x)$ is of finite rank, there exist a sequence $\{y_{n,i}\}_{i=m_{n-1}+1}^{m_n} \subset \mathcal{X}$ and a total sequence of row finite matrix of functionals $\{g_{n,i}\}_{i=m_{n-1}+1}^{m_n} \subset \mathcal{X^*}$ such that
\begin{align*} S_n(x)=\sum_{i=m_{n-1}+1}^{m_n}g_{n,i}(x)y_{n,i}, \mbox{for all} ~~x \in \mathcal{X}, ~~n \in \mathbb{N}.
\end{align*} Define sequences $\{x_n\}\subset \mathcal{X}$ and $\{h_{n,i}\}\underset{n \in \mathbb{N}}{_{i=1,2,3,...,m_n}}\subset \mathcal{X}^*,$ where $\{m_n\}$ is an increasing sequence of positive integers, by
\begin{align*} x_i=y_{n,i}, ~i=m_{n-1}+1,...m_n; n=1,2,3...
\end{align*} and
\begin{eqnarray*}
 h_{n,i}=\begin{cases} 0, ~~~ \mbox{for}~~~  i=1,2,...,m_{n-1}  \\ g_{n,i},~~~ \mbox{for} ~~~ i=m_{n-1}+1,...,m_n.
\end{cases}
\end{eqnarray*}Then $x_n\neq 0,$ so for each $x \in \mathcal{X}$ and $n \in \mathbb{N},$ we get
\begin{align}\label{aa} \lim\limits_{n \to \infty} \sum_{i=1}^{m_n}h_{n,i}(x)x_i = \lim\limits_{n \to \infty}S_n(x)=K(x).
\end{align} Let $x \in \mathcal{X}$ be such that $h_{n,i}(x)=0, \mbox{for all}~i=1,2,...m_n, ~~n \in \mathbb{N}.$ Then by equation (\ref{aa}) $K(x)=0.$ Thus by Lemma \ref{KL} there exist an associated Banach space $\mathcal{X}_d=\{\{h_{n,i}\}\underset{n \in \mathbb{N}}{_{i=1,2,3,...,m_n}},x \in \mathcal{X}\}$ with norm given by $\| \{h_{n,i}\}\underset{n \in \mathbb{N}}{_{i=1,2,3,...,m_n}}\|_{\mathcal{X}_d}=\|x\|_\mathcal{X}, ~~\mbox{for all}~ x \in \mathcal{X}.$ Hence $(\{h_{n,i}\}, \{x_n\})$ is an approximative $K$-atomic decomposition for $\mathcal{X}$ with respect to $\mathcal{X}_d.$
 \end{proof}
Next, we give an example of an approximative $K$-atomic decomposition for $\mathcal{X}$ which is not an approximative atomic decomposition for $\mathcal{X}.$
\begin{example}
Let $\mathcal{X}=c_0 \ and \ \mathcal{X}_d=l_\infty$. Let $\lbrace x_n\rbrace\subset \mathcal{X}$ be the sequence of standard unit vectors in $\mathcal{X}$ and $\lbrace h_{n,i}\rbrace\subseteq \mathcal{X}^*$ be such that for $x=\lbrace \alpha_n\rbrace\in \mathcal{X},h_{n,1}(x)=0,h_{n,2}(x)=\alpha_2,...,h_{n,i}(x)=\alpha_n,....$ It is clear that $\lim\limits_{n \to \infty}\sum\limits_{i=1}^{m_n}h_{n,i}(x)x_i$ converges for $x\in \mathcal{X}$. Define $ K: \mathcal{X} \longrightarrow \mathcal{X} $ by $K(x)=\lim\limits_{n \to \infty}\sum\limits_{i=1}^{m_n}h_{n,i}(x)x_i$ , $x\in \mathcal{X}.$ Then $\lbrace h_{n,i}(x)\rbrace\in \mathcal{X}_d$  is such that $(\{x_n\},\{h_{n,i}(x)\})$ is an approximative $K$-atomic decomposition for $\mathcal{X}$ with respect to $\mathcal{X}_d$. But $(\{x_n\},\{h_{n,i}\})$ is not an approximative atomic decomposition for $\mathcal{X}.$
\end{example}

Next, we give various methods for the construction of approximative $K$-atomic decompositions for $\mathcal{X}.$
\begin{theorem}\label{TC}
Let $(\{x_n\},\{h_{n,i}\}\underset{n \in \mathbb{N}}{_{i=1,2,3,...,m_n}})$  be an approximative atomic decomposition for $\mathcal{X}$ with respect to $\mathcal{X}_d$ with bounds $A$ and $B$. Let $K \in L(\mathcal{X})$ with $K\neq 0$. Then $(\{Kx_n\},\{h_{n,i}\}\underset{n \in \mathbb{N}}{i=1,2,3,...,m_n})$ is an approximative $K$-atomic decomposition for $\mathcal{X}$ with respect to $\mathcal{X}_d$ with bounds $\frac{A}{\|K\|}$ and $B.$
\end{theorem}
\begin{proof} Since $(\{x_n\},\{h_{n,i}\}\underset{n \in \mathbb{N}}{_{i=1,2,3,...,m_n}})$ is an approximative atomic decomposition for $\mathcal{X}$ with respect to $\mathcal{X}_d$ with bounds $A$ and $B.$ So for each $x\in \mathcal{X},$ we have $x=\lim\limits_{n \to \infty}\sum\limits_{i=1}^{m_n}h_{n,i}(x)x_i.$ This implies $K(x)=\lim\limits_{n \to \infty}\sum\limits_{i=1}^{m_n}h_{n,i}(x)K(x_i).$
Also, we have  $\parallel K(x) \parallel_\mathcal{X} ~\leq~ \parallel K \parallel \parallel x\parallel_\mathcal{X}$, for all $x\in \mathcal{X}.$ This gives
\begin{eqnarray*}
\dfrac{A}{\parallel K\parallel}\parallel K(x)\parallel_\mathcal{X}~\leq~ \parallel\lbrace h_{n,i}(x)\rbrace\Vert_{\mathcal{X}_d}\leq B\parallel x\parallel_\mathcal{X},~  \mbox{for all} \ x\in \mathcal{X}.
\end{eqnarray*}
\end{proof}
\begin{theorem}
Let $(\{x_n\},\{h_{n,i}\}\underset{n \in \mathbb{N}}{_{i=1,2,3,...,m_n}})$  be an approximative atomic decomposition for $\mathcal{X}$ with respect to $\mathcal{X}_d$ with bounds $A$ and $B$. Let $K \in L(\mathcal{X})$ with $K\neq 0$. Then $(\{x_n\},\{K^*h_{n,i}\}\underset{n \in \mathbb{N}}{_{i=1,2,3,...,m_n}})$ is an approximative $K$-atomic decomposition for $\mathcal{X}$ with respect to $\mathcal{X}_d$ with bounds $A$ and $B\|K\|.$
\end{theorem}
\begin{proof} Can be easily proved.
\end{proof}
\begin{theorem}
Let $(\{x_n\},\{h_{n,i}\}\underset{n \in \mathbb{N}}{_{i=1,2,3,...,m_n}})$  be an approximative $K$-atomic decomposition for $\mathcal{X}$ with respect to $\mathcal{X}_d$ with bounds $A$ and $B$ and let $T\in L(\mathcal{X})$ with $T\neq0$. Then $(\{Tx_n\},\{h_{n,i}\})$ is an approximative $TK$-atomic decomposition for $\mathcal{X}$ with respect to $ \mathcal{X}_d $ with bounds $\frac{A}{\|T\|}$ and $B.$
\end{theorem}
\begin{proof}
Construction of proof is similar to theorem \ref{TC}.
\end{proof}
\begin{theorem}
Let $(\{x_n\},\{h_{n,i}\}\underset{n \in \mathbb{N}}{_{i=1,2,3,...,m_n}})$  be an approximative $K$-atomic decomposition for $\mathcal{X}$ with respect to $\mathcal{X}_d$ with bounds $A$ and $B$ and let $T\in L(\mathcal{X})$ with $\|T\|\neq 0$ Then $(\{x_n\},\{T^*h_{n,i}\})$ is an approximative $KT$-atomic decomposition for $\mathcal{X}$ with respect to $ \mathcal{X}_d $ with bounds $A$ and $B\|T\|.$
\end{theorem}
\begin{proof}  Obvious
\end{proof}

\begin{theorem}
If $(\{x_n\},\{h_{n,i}\}\underset{n \in \mathbb{N}}{_{i=1,2,3,...,m_n}})$ be an approximative $K$-atomic decomposition for $\mathcal{X}$ with respect to $\mathcal{X}_d$ and  $K$ has pseudo inverse $K^\dagger$, then there exists $\lbrace g_{n,i} \rbrace \subseteq \mathcal{X}^*$ such that $(\{x_n\},\{g_{n,i}\})$ is an approximative $K$-atomic decomposition for $\mathcal{X}$ with respect to $\mathcal{X}_d$ with bounds $A$ and $B\|K\|^2$
\end{theorem}
\begin{proof} Since $(\{x_n\},\{h_{n,i}\}\underset{n \in \mathbb{N}}{_{i=1,2,3,...,m_n}})$  is an approximative $K$-atomic decomposition for $\mathcal{X}$ with respect to $\mathcal{X}_d,$ then for each $x \in \mathcal{X}$ we have
\begin{eqnarray*}
A\| K(x)\|_{\mathcal{X}} \leq \| \{ h_{n,i}(x)\}\|_{\mathcal{X}_d} \leq B\| x\|_{\mathcal{X}},~x \in \mathcal{X}.
\end{eqnarray*}
Also, for each $x\in \mathcal{X}$, we have
\begin{eqnarray*}
K(x)=K(K^\dagger K(x))=\lim\limits_{n \to \infty}\sum\limits_{i=1}^{m_n}h_{n,i}(K^\dagger K(x))x_i
=\lim\limits_{n \to \infty}\sum\limits_{i=1}^{m_n}((K^\dagger K)^*(h_{n,i})(x))x_i.
\end{eqnarray*}
For each $n \in \mathbb{N}$, define $g_{n,i}={(K^\dagger K)}^*(h_{n,i})$. Then
\begin{equation*}
 \|K(x)\|_\mathcal{X}=\|K(K^\dagger K(x))\|_\mathcal{X} \leq \dfrac{1}{A}\| \{h_{n,i}(K^\dagger K(x)) \} \|_{\mathcal{X}_d}=\dfrac{1}{A}\| \{ g_{n,i}(x)\} \|_{\mathcal{X}_d},~ x \in \mathcal{X}
\end{equation*}
and
\begin{equation*}
\| \{g_{n,i}(x)\}\|_{\mathcal{X}_d}=\|\{ h_{n,i}(K^\dagger K(x))\}\|_{\mathcal{X}_d} \leq B\| K^\dagger\|\| K\|\|x\|_\mathcal{X},~ x \in \mathcal{X}.
\end{equation*}
Hence, we conclude that $(\{x_n\},\{g_{n,i}\})$ is an approximative $K$-atomic decomposition for $\mathcal{X}$ with respect to $\mathcal{X}_d$.
\end{proof}
\section{Approximative $\mathcal{X}_d$-frame}
Casazza et al. \cite{CCS} defined and studied $\mathcal{X}_d$-Bessel sequences and $\mathcal{X}_d$-frames in Banach spaces. Later on Stoeva \cite{DT} studied perturbation of $\mathcal{X}_d$-Bessel sequences, $\mathcal{X}_d$-frames, atomic decomposition and $\mathcal{X}_d$-Riesz bases in separable Banach spaces. We have generalized this concept and defined approximative $\mathcal{X}_d$-Bessel sequences and  approximative $\mathcal{X}_d$-frames in Banach spaces. We begin this section with the following definitions:
\begin{definition}
A sequence $ \{h_{n,i} \}\underset{n \in \mathbb{N}}{_{i=1,2,3,...,m_n}} \subseteq \mathcal{X}^*,$ where $\{m_n\}$ is an increasing sequence of positive integers, is called an approximative \emph{$\mathcal{X}_d$-frame} for $\mathcal{X}$ if
\begin{enumerate}[(a)]
\item $\{h_{n,i}(x) \}\underset{n \in \mathbb{N}}{_{i=1,2,3,...,m_n}} \in \mathcal{X}_d$, for all $ x\in \mathcal{X}.$
\item There exist constants $A$ and $B$ with $0 < A \leq B <\infty$ such that
\begin{eqnarray}\label{XD}
 A\|x \|_\mathcal{X}~\leq \| \{h_{n,i}(x) \}\underset{n \in \mathbb{N}}{_{i=1,2,3,...,m_n}} \|_{\mathcal{X}_d}~\leq B \|x \|_\mathcal{X},~\text{for all} \ x\in \mathcal{X}.
\end{eqnarray}
\end{enumerate}
\end{definition}
\noindent The constants $A$ and $B$ are called approximative \emph{$\mathcal{X}_d$-frame} bounds.
If atleast (a) and the upper bound condition in (\ref{XD}) are satisfied, then $\{h_{n,i}\}$ is called an approximative \emph{$\mathcal{X}_d$-Bessel sequence} for $\mathcal{X}.$\\
One may note that if $\{f_n\}$ is an $\mathcal{X}_d$-frame for $\mathcal{X},$ then for $\{h_{n,i} \}=f_i,~i=1,2,3,...,n;~ n \in \mathbb{N}$, $\{h_{n,i}\}$ is an approximative $\mathcal{X}_d$-frame for $\mathcal{X}.$ Also, note that if $\{f_n\}$ is an $\mathcal{X}_d$-Bessel sequence for $\mathcal{X},$ then for $\{h_{n,i} \}=f_i,~i=1,2,3,...,n;~ n \in \mathbb{N}$, $\{h_{n,i}\}$ is an approximative $\mathcal{X}_d$-Bessel sequence for $\mathcal{X}.$\\
In the next two results, we give necessary conditions under which an approximative $\mathcal{X}_d$-frame gives rise to a bounded operator $K$ with respect to which there is an approximative $K$-atomic decomposition for $\mathcal{X}$.
\begin{theorem}
Let $\{h_{n,i} \}\underset{n \in \mathbb{N}}{_{i=1,2,3,...,m_n}} \subseteq \mathcal{X}^* $ be an approximative $\mathcal{X}_d$-frame for $\mathcal{X}$ with bounds $A$ and $B$. Let $\lbrace x_n \rbrace \subseteq \mathcal{X}$  with $\sup\limits_{1 \leq n < \infty}\Vert x_n\Vert < \infty$ and let $\lim\limits_{n \to \infty}\sum\limits_{i=1}^{m_n}|h_{n,i}(x)| < \infty,$ for all $x \in  \mathcal{X}$. Then there exists an operator $K \in L(\mathcal{X})$ such that $(\{x_n\},\{h_{n,i}\})$ is an approximative $K$-atomic decomposition for $\mathcal{X}$ with respect to $\mathcal{X}_d$.
\end{theorem}
\begin{proof}
Since $\{h_{n,i} \}\underset{n \in \mathbb{N}}{_{i=1,2,3,...,m_n}} \subseteq \mathcal{X}^* $ is an approximative $\mathcal{X}_d$-frame for $\mathcal{X}$ with $\sup\limits_{1 \leq n < \infty}\Vert x_n\Vert < \infty$ and $\lim\limits_{n \to \infty}\sum\limits_{i=1}^{m_n}|h_{n,i}(x)| < \infty.$ Then, by Theorem \ref{T1}, we have $\lim\limits_{n \to \infty}\sum\limits_{i=1}^{m_n}h_{n,i}(x)x_i$ exist for all $x\in \mathcal{X}, ~n \in \mathbb{N}.$
\\
Define $K:\mathcal{X}\longrightarrow \mathcal{X}$ by $K(x)=\lim\limits_{n \to \infty}\sum\limits_{i=1}^{m_n}h_{n,i}(x)x_i$, $x\in \mathcal{X}$. Then $K$ is a bounded linear operator such that
\begin{equation*}
\Vert K(x)\Vert_\mathcal{X}\leq\sup\limits_{1\leq n<\infty}\Vert\sum_{i=1}^{m_n}h_{n,i}(x)x_i\Vert_\mathcal{X}\leq C\Vert x\Vert_\mathcal{X},
\end{equation*}
where $C= \sup\limits_{1\leq n <\infty}\sum\limits_{i=1}^{m_n}h_{n,i}(x)x_i$. Thus
\begin{equation*}
\dfrac{A}{C}\Vert K(x)\Vert_\mathcal{X}\leq\parallel\lbrace h_{n,i}(x)\rbrace\Vert_{\mathcal{X}_d}\leq B\parallel x\parallel_\mathcal{X},~\mbox{for all}~ x \in \mathcal{X}.
\end{equation*}
Hence, $(\{x_n\},\{h_{n,i}\})$ is an approximative $K$-atomic decomposition for $\mathcal{X}$ with respect to $\mathcal{X}_d$ with bounds $\dfrac{A}{C}$ and $B$.
\end{proof}
\begin{theorem}
Let $\{h_{n,i} \}\underset{n \in \mathbb{N}}{_{i=1,2,3,...,m_n}} \subseteq \mathcal{X}^* $ be an approximative $\mathcal{X}_d$-frame with bounds $A$,~$B$ and let $\lbrace x_n\rbrace \subseteq \mathcal{X}$. Let $T:\mathcal{X}_d\longrightarrow \mathcal{X}$ given by $T(\lbrace h_{n,i} \rbrace)=\lim\limits_{n \to \infty}\sum\limits_{i=1}^{m_n}h_{n,i}x_i$ be a well defined operator. Then, there exists a linear operator $K \in L(\mathcal{X})$ such that $(\{x_n\},\{h_{n,i}\})$ is an approximative $K$-atomic decomposition for $\mathcal{X}$ with respect to $\mathcal{X}_d.$
\end{theorem}
\begin{proof}
Define $U:\mathcal{X}\longrightarrow \mathcal{X}_d$ by $U(x)=\lbrace h_{n,i}(x)\rbrace$, $x\in \mathcal{X}$. Then $U$ is well defined and $\Vert U\Vert\leq B$. Take $K=TU$. Then $K(x)=\lim\limits_{n \to \infty}\sum\limits_{i=1}^{m_n}h_{n,i}(x)x_i, ~ x \in \mathcal{X}$. Therefore, by uniform boundedness principle, we have
\begin{equation*}
\Vert K(x)\Vert_\mathcal{X}\leq\sup\limits_{1\leq n<\infty}\Vert \sum\limits_{i=1}^{m_n}h_{n,i}(x)x_i\Vert_\mathcal{X}\leq C\Vert x\Vert_\mathcal{X},~x\in \mathcal{X},
\end{equation*}
where $C=\sup\limits_{1\leq n<\infty}\Vert \sum\limits_{i=1}^{m_n}h_{n,i}(x)x_i\Vert_\mathcal{X}$. Thus, we have
\begin{equation*}
\dfrac{A}{C}\Vert K(x)\Vert~\leq~\parallel\lbrace h_{n,i}(x)\rbrace\Vert\leq B\parallel x\parallel,~ \mbox{for all} \ x \in \mathcal{X}.
\end{equation*}
Hence $(\{x_n\},\{h_{n,i}\})$ is an approximative $K$-atomic decomposition for $\mathcal{X}$ with respect to $\mathcal{X}_d$ with bounds $\dfrac{A}{C}$ and $B$.
\end{proof}

Next, we give the existence of an approximative $K$-atomic decomposition from an approximative $\mathcal{X}_d$ Bessel sequence.
\begin{theorem}
Let $\mathcal{X}$ be a reflexive Banach space and $\mathcal{X}_d$ be a BK-space which has a sequence of canonical unit vectors $\lbrace e_n\rbrace$ as a basis. Let $\{h_{n,i} \}\underset{n \in \mathbb{N}}{_{i=1,2,3,...,m_n}} \subseteq \mathcal{X}^* $ be an approximative $\mathcal{X}_d$-Bessel sequence with bound $B$ and let $\lbrace x_n\rbrace \subseteq \mathcal{X}$. If $\lbrace h(x_n)\rbrace \in (\mathcal{X}_d)^*$ for all $h\in \mathcal{X}^*$, then there exists a bounded linear operator $K \in L(\mathcal{X})$ such that $(\{x_n\},~\{h_{n,i} \}\underset{n \in \mathbb{N}}{_{i=1,2,3,...,m_n}})$ is an approximative $K$-atomic decomposition for $\mathcal{X}$ with respect to $\mathcal{X}_d.$
\end{theorem}
\begin{proof}
Clearly $U:\mathcal{X} \longrightarrow \mathcal{X}_d$ given by $U(x)=\lbrace h_{n,i}(x)\rbrace,~ x\in \mathcal{X}$ is well defined. Define a map $R:\mathcal{X}^*\longrightarrow (\mathcal{X}_d)^*$ by $R(h)=\lbrace h(x_n)\rbrace, ~x\in \mathcal{X}$. Then, its adjoint $R^*:(\mathcal{X}_d)^{**}\longrightarrow \mathcal{X}^{**}$ is given by $R^*(e_j)(h)=e_j(R(h))=h(x_j)$. Let $T=(R^*)\vert_{\mathcal{X}_d}$ and $\lbrace h_{n,i} \rbrace \in \mathcal{X}_d$. Then
\begin{eqnarray*}
T(\lbrace h_{n,i}\rbrace)=\lim\limits_{n \to \infty}\sum\limits_{i=1}^{m_n}h_{n,i}T(e_i)=\lim\limits_{n \to \infty}\sum\limits_{i=1}^{m_n}h_{n,i}x_i.
\end{eqnarray*}
But $\lbrace h_{n,i}(x)\rbrace \in \mathcal{X}_d$. So $T(\lbrace h_{n,i}(x)\rbrace)=\lim\limits_{n \to \infty}\sum\limits_{i=1}^{m_n}h_{n,i}(x)x_i$.
Take $K = TU $.
Then $K\in L(\mathcal{X}) \ \text{and} \ K(x)=\lim\limits_{n \to \infty}\sum\limits_{i=1}^{m_n}h_{n,i}(x)x_i.$
Moreover, $T$ is a bounded linear operator such that $\Vert K(x)\Vert\leq\Vert T\Vert\Vert\lbrace h_{n,i}(x)\rbrace\Vert$.
Hence
\begin{equation*}
 \dfrac{1}{\Vert T\Vert}\Vert K(x)\Vert\leq\Vert\lbrace h_{n,i}(x)\rbrace\Vert\leq B\Vert x\Vert, ~x\in \mathcal{X}
\end{equation*}
\end{proof}
Next, we construct an approximative $K^*$-atomic decomposition for $\mathcal{X}^*$ from a given approximative $K$-atomic decomposition for $\mathcal{X}.$
\begin{theorem}
Let $\mathcal{X}_d$ be a BK-space with dual $(\mathcal{X}_d)^*$ and let $\mathcal{X}_d$ and $(\mathcal{X}_d)^*$ have sequences of canonical unit vectors $\lbrace e_n\rbrace  \ \text{and}\ \   \lbrace v_{n}\rbrace$ respectively as bases. Let $(\{x_n\},\{h_{n,i} \}\underset{n \in \mathbb{N}}{_{i=1,2,3,...,m_n}})$ be an approximative $K$-atomic decomposition for $\mathcal{X}$ with respect to $\mathcal{X}_d$. Let  $S:\mathcal{X}_d\longrightarrow \mathcal{X}$ given by $S(\lbrace h_{n,i}\rbrace)=\lim\limits_{n \to \infty}\sum\limits_{i=1}^{m_n}h_{n,i}x_i$ be a well defined mapping. Then, $(\{h_{n,i} \}\underset{n \in \mathbb{N}}{_{i=1,2,3,...,m_n}},\pi(x_n))$ is an approximative $K^*$-atomic decomposition for $\mathcal{X}^*$ with respect to $(\mathcal{X}_d)^*$.
\end{theorem}
\begin{proof} Since $(\{x_n\},\{h_{n,i} \}\underset{n \in \mathbb{N}}{_{i=1,2,3,...,m_n}})$ is an approximative $K$-atomic decomposition for $\mathcal{X}$ with respect to $\mathcal{X}_d,$ so for each $x\in \mathcal{X}$, $K(x)=\lim\limits_{n \to \infty}\sum\limits_{i=1}^{m_n}h_{n,i}(x)x_i.$ Thus
$h(K(x))=\lim\limits_{n \to \infty}\sum\limits_{i=1}^{m_n}h_{n,i}(x)h(x_i).$ Therefore, by Theorem $\ref{T1}$ we have $\lim\limits_{n \to \infty}\sum\limits_{i=1}^{m_n}h(x_i)h_{n,i}$ exist for all $h\in \mathcal{X}^*.$ Also, for $x\in \mathcal{X}$, we compute
\begin{equation*}
(K^*(h))(x)=h(\lim\limits_{n \to \infty}\sum\limits_{i=1}^{m_n}h_{n,i}(x)x_i)=\lim\limits_{n \to \infty}\sum\limits_{i=1}^{m_n}h(x_i)h_{n,i}(x).
\end{equation*}
This gives $K^*(h)=\lim\limits_{n \to \infty}\sum\limits_{i=1}^{m_n}h(x_i)h_{n,i}$, for $h\in \mathcal{X}^*$. Note that
\\
 $S^*(h)(e_j) = h(S(e_j)) = h(x_j), h\in \mathcal{X}^*$. So, $S^*(h) = \lbrace h(x_n)\rbrace$ and $\lbrace h(x_n)\rbrace = \lbrace h(S(e_n))\rbrace \in (\mathcal{X}_d)^*, h\in \mathcal{X}^*$.
Also
\begin{eqnarray*}
\Vert\lbrace h(x_n)\rbrace\Vert_{(\mathcal{X}_d)^*} = \Vert S^*(h)\Vert\leq\Vert S\Vert\Vert h\Vert_{\mathcal{X}^*},~ h\in \mathcal{X}^*.
\end{eqnarray*}
Define $R:\mathcal{X} \longrightarrow \mathcal{X}_d$ by $R(x) = \lbrace h_{n,i}(x)\rbrace, x\in \mathcal{X}$. Then, $R^*(v_j)(x) = v_j(R(x)) = h_{j,i}(x)$, $x\in \mathcal{X}$. So, $R^*(v_j)=h_{j,i}$, for all $j\in \mathbb{N}$ and for $\lbrace g_{n,i} \rbrace \in (\mathcal{X}_d)^*$ we have
 \begin{eqnarray*}
 R^*(\lbrace g_{n,i}\rbrace)=R^*(\lim\limits_{n \to \infty}\sum\limits_{i=1}^{m_n}g_{n,i}(x)v_i)
 =\lim\limits_{n \to \infty}\sum\limits_{i=1}^{m_n}g_{n,i}(x)h_{n,i}.
 \end{eqnarray*}
 Therefore, we have
 \begin{eqnarray*}
  R^*S^*(h)=R^*(\lbrace h(x_i)\rbrace)=\lim\limits_{n \to \infty}\sum\limits_{i=1}^{m_n}h(x_i)h_{n,i}, ~h \in \mathcal{X}^*.
\end{eqnarray*}
 Note that, $K^*=R^*S^*$ and so
 \begin{eqnarray*}
 \Vert K^*(h)\Vert_{\mathcal{X}^*}=\Vert R^*S^*(h)\Vert_{\mathcal{X}^*}\leq\Vert R^*\Vert\Vert\lbrace h(x_n)\rbrace\Vert_{(\mathcal{X}_d)^*}, ~h\in \mathcal{X}^*.
 \end{eqnarray*}
 This gives
\begin{equation}
\dfrac{1}{\Vert R^*\Vert}\Vert K^*(h)\Vert_{\mathcal{X}^*}\leq\Vert\lbrace h(x_n)\rbrace\Vert_{(\mathcal{X}_d)^*}\leq\Vert S\Vert\Vert h\Vert_{\mathcal{X}^*} ,~ h \in \mathcal{X}^*.
\end{equation}
Hence, $(\{h_{n,i} \}\underset{n \in \mathbb{N}}{_{i=1,2,3,...,m_n}},\pi(x_n))$ is an approximative $K^*$-atomic decomposition for $\mathcal{X}^*$ with respect to $(\mathcal{X}_d)^*$.
\end{proof}
Next, we give the following result characterizing the class of approximative $K$-atomic decompositions.
\begin{theorem}
Let $(\{x_n\},\{h_{n,i} \}\underset{n \in \mathbb{N}}{_{i=1,2,3,...,m_n}})$  be an approximative $K$-atomic decomposition for $\mathcal{X}$ with respect to $\mathcal{X}_d$ with bounds $A$ and $B$. Let $T:\mathcal{X}_d\longrightarrow \mathcal{X}$ given by $T(\lbrace h_{n,i} \rbrace) =  \lim\limits_{n \to \infty}\sum\limits_{i=1}^{m_n}h_{n,i}x_i$ is well defined for $\lbrace h_{n,i} \rbrace \in \mathcal{X}_d$ and let $U:\mathcal{X} \longrightarrow \mathcal{X}_d$ be the mapping given by $U(x)=\lbrace h_{n,i}(x)\rbrace$. If $K$ is invertible, then the following statements are equivalent.
\begin{enumerate}[(a)]
\item $T$ is the pseudo inverse of $U$.
\item $(\{x_n\},\{h_{n,i} \}\underset{n \in \mathbb{N}}{_{i=1,2,3,...,m_n}})$ is an approximative atomic decomposition for $\mathcal{X}$ with respect to $\mathcal{X}_d.$
\item $T$ is a linear extension of $U^{-1}:U(\mathcal{X})\longrightarrow \mathcal{X}$.
\item $U(\mathcal{X})$ is a complemented subspace of $\mathcal{X}_d.$
\item $KerT$ is a complemented subspace of $\mathcal{X}_d$ and $T$ is surjective.
\end{enumerate}
\end{theorem}
\begin{proof}
$(a)\Rightarrow(b)$
By hypothesis, $\lbrace x\in \mathcal{X}:h_{n,i}(x)=0,~ for \ all \ n \in\NN\rbrace=\lbrace 0\rbrace$. So, $KerU=\lbrace 0\rbrace$. Since $T$ is the pseudo inverse of $U$, by Lemma \ref{L1} there exists a continuous projection operator $\theta:\mathcal{X} \longrightarrow \mathcal{X}$ such that $TU=I_\mathcal{X}-\theta$ and $kerU=\theta(\mathcal{X})$. Thus, for each $x\in \mathcal{X},$ we have
\begin{eqnarray*}
TU(x)=(I_\mathcal{X}-\theta)(x)=x, ~x\in \mathcal{X}.
\end{eqnarray*}
 Hence, for every $x\in \mathcal{X}$, $\lim\limits_{n \to \infty}\sum\limits_{i=1}^{m_n}h_{n,i}(x)x_i=x$.
\\
$(b)\Rightarrow(a) \ $For $x\in \mathcal{X}$, we have
\begin{eqnarray*}
UTU(x)=UT(\lbrace h_{n,i}(x) \rbrace)=U(\lim\limits_{n \to \infty}\sum\limits_{i=1}^{m_n}h_{n,i}(x)x_i)=U(x).
\end{eqnarray*}
 Hence, $UTU=U$.
\\
$(c)\Rightarrow(b) \ \  $If $T$ is a linear extension of $U^{-1}:U(\mathcal{X})\longrightarrow \mathcal{X}$, then $TU:\mathcal{X}\longrightarrow \mathcal{X}$ is the identity map on $\mathcal{X}$. So, $TU(x)=x$ and $\lim\limits_{n \to \infty}\sum\limits_{i=1}^{m_n}h_{n,i}(x)x_i=x$.
\\
$(c)\Rightarrow(a) \  $Obvious, since $UTU=UI_\mathcal{X}=U.$
\\
(d)$\Rightarrow$(b) \ Suppose $\mathcal{X}_d=U(\mathcal{X})\oplus G$, where $G$ is a closed subspace of $\mathcal{X}_d$. Let \emph{P} be a projection of $\mathcal{X}_d$ onto $U(\mathcal{X})$ along $G$.\\Then, $\emph{P}(\lbrace h_{n,i} \rbrace)=\lbrace g_{n,i}(\lim\limits_{n \to \infty}\sum\limits_{i=1}^{m_n} h_{n,i}x_i)\rbrace$, for all $\lbrace h_{n,i}\rbrace \in \mathcal{X}_d$. Therefore
\begin{eqnarray*}
U^{-1}\circ \emph{P}(\lbrace h_{n,i} \rbrace)&=&U^{-1}\lbrace g_{n,i}(\lim\limits_{n \to \infty}\sum\limits_{i=1}^{m_n} h_{n,i}x_i)\rbrace=\lim\limits_{n \to \infty}\sum\limits_{i=1}^{m_n} h_{n,i}x_i
\\
&=&T(\lbrace h_{n,i} \rbrace),~ for \ all \ \lbrace h_{n,i} \rbrace\in \mathcal{X}_d.
\end{eqnarray*}
This gives, T=$U^{-1}\circ \emph{P}$ and
\begin{eqnarray*}
T(\lbrace h_{n,i}(x)\rbrace)=U^{-1}\circ \emph{P}(\lbrace h_{n,i}(x)\rbrace)=U^{-1}(\lbrace h_{n,i}(x)\rbrace.
\end{eqnarray*}
Hence, $x=\lim\limits_{n \to \infty}\sum\limits_{i=1}^{m_n} h_{n,i}(x)x_i$, for all $x\in \mathcal{X}$.
\\
(b)$\Rightarrow$(d) \ Obvious.
\\
(e)$\Rightarrow$(b) Let $\mathcal{X}_d=kerT\oplus M$, where $M$ is a closed subspace of $\mathcal{X}_d.$ Take $\Upsilon=kerT\oplus U(\mathcal{X}).$ Let $Q:\mathcal{X}_d\longrightarrow M$ be a projection from $\mathcal{X}_d$ onto $M$ along $kerT$. Define $L:\mathcal{X}_d\longrightarrow \Upsilon$ by $L(\alpha)=(\alpha -Q(\alpha),UT(\alpha))$, for $\alpha=\lbrace h_{n,i} \rbrace \in \mathcal{X}_d.$ Let $L(\alpha)=0$. This gives $Q(\alpha)= \alpha.$ So $\alpha \in M.$ Let $UT(\alpha)=0.$ Then
\begin{eqnarray*}
U(\lim\limits_{n \to \infty}\sum\limits_{i=1}^{m_n} h_{n,i}x_i)=\lbrace g_{n,i}(\lim\limits_{n \to \infty}\sum\limits_{i=1}^{m_n} h_{n,i}x_i)\rbrace=0,~ for \ n\in \mathbb{N}.
\end{eqnarray*}
 This gives $\lim\limits_{n \to \infty}\sum\limits_{i=1}^{m_n} h_{n,i}x_i=0$ and so, $\alpha \in kerT$. Thus, $\alpha \in kerT \cap M=\lbrace 0\rbrace$. Hence, $L$ is one-one.
\\
Let $(\alpha_0,U(x)) \in kerT\oplus U(\mathcal{X}),$ for $\alpha_0 \in kerU$ and $U(x)\in U(\mathcal{X}).$
\\
Since, $T$ is onto, for each $x\in \mathcal{X}$, there exists $\beta\in \mathcal{X}_d$ such that $T(\beta)$ = $x$ and this gives $UT(\beta)=U( x ).$ Take $\alpha=\alpha_0+Q(\beta)$. Then $Q(\alpha)=Q(\alpha_0)+Q^2(\beta)=Q(\beta)$ and $\alpha_0=\alpha-Q(\alpha)$. Also, we have
\begin{eqnarray}
UT(\alpha)=UT(\alpha-\alpha_0)=UT(Q(\beta))=UT(\beta)=U( x ).
\end{eqnarray}
Thus $L(\alpha)=(\alpha_0,UT(x))$
 and $L$ is an isomorphism from $\mathcal{X}_d$ onto $\Upsilon.$ So, there is a projection $\emph{P}=UT:\mathcal{X}_d\longrightarrow U(\mathcal{X})$ onto $U(\mathcal{X})$ along $ kerT $. This gives
 \begin{eqnarray*}
 U^{-1}\circ \emph{P}=T \ \ and \ \ U^{-1}\circ \emph{P}(\lbrace h_{n,i}(x)\rbrace)=T(\lbrace h_{n,i}(x)\rbrace).
 \end{eqnarray*}
 Finally, we compute
 \begin{eqnarray*}
 U^{-1}(\lbrace h_{n,i}(x) \rbrace)=\lim\limits_{n \to \infty}\sum\limits_{i=1}^{m_n}h_{n,i}(x)x_i \ \ and \ \  x=\lim\limits_{n \to \infty}\sum\limits_{i=1}^{m_n} h_{n,i}(x)x_i.
 \end{eqnarray*}
Therefore, $(\{x_n\},\{h_{n,i} \}\underset{n \in \mathbb{N}}{_{i=1,2,3,...,m_n}})$ is an approximative atomic decomposition for $\mathcal{X}$ with respect to $\mathcal{X}_d$.
\\
(b)$\Rightarrow$(e) Obvious.
\end{proof}
In the following result, we prove a duality type approximative $K$-atomic decomposition for $\mathcal{X}$.
\begin{theorem}
Let $\mathcal{X}_d$ be a reflexive BK-space with its dual $(\mathcal{X}_d)^*$ and let sequences of canonical unit vectors $\lbrace e_n\rbrace \ and \  \lbrace v_n\rbrace$ be bases for $\mathcal{X}_d$ and $(\mathcal{X}_d)^*$, respectively. Let $(\{h_{n,i} \}\underset{n \in \mathbb{N}}{_{i=1,2,3,...,m_n}},\pi(x_n))$ be an approximative $K$-atomic decomposition for $\mathcal{X}^*$ with respect to $(\mathcal{X}_d)^*.$ If $S:({\mathcal{X}_d})^*\longrightarrow \mathcal{X}^*$ given by $S(\lbrace d_i\rbrace)=\lim\limits_{n \to \infty}\sum\limits_{i=1}^{m_n}d_ih_{n,i}$  is well defined for $\lbrace d_i\rbrace \in \mathcal{X}_d^*$, then there exists a linear operator $L\in L(\mathcal{X})$ such that $(\{x_n\},\{h_{n,i} \}\underset{n \in \mathbb{N}}{_{i=1,2,3,...,m_n}})$ is an approximative $L$-atomic decomposition for $\mathcal{X}$ with respect to $\mathcal{X}_d.$
\end{theorem}
\begin{proof} Since $(\{h_{n,i} \}\underset{n \in \mathbb{N}}{_{i=1,2,3,...,m_n}},\pi(x_n))$ is an approximative $K$-atomic decomposition for $\mathcal{X}^*$ with respect to $(\mathcal{X}_d)^*.$
For $h\in \mathcal{X}^*$, we have $K(h)=\lim\limits_{n \to \infty}\sum\limits_{i=1}^{m_n}h(x_i)h_{n,i}$. Also, by Theorem \ref{T1}
 we have $\lim\limits_{n \to \infty}\sum\limits_{i=1}^{m_n}h_{n,i}(x)x_i$ exist, for all $x\in \mathcal{X}.$
Define$ L:\mathcal{X} \longrightarrow \mathcal{X}$  by \ $L(x) = \lim\limits_{n \to \infty}\sum\limits_{i=1}^{m_n}h_{n,i}(x)x_i,~ x\in \mathcal{X}.$ Note that $S(v_n)=h_{n,i},~n\in \NN$ and for $x\in \mathcal{X}$, the linear bounded operator $S^*:\mathcal{X}^{**}\longrightarrow (\mathcal{X}_d)^{**}$ satisfies
\begin{eqnarray*}
S^*(\pi(x))(v_n)=\pi(x)S(v_n) = h_{n,i}(x).
\end{eqnarray*}
So, $\lbrace h_{n,i}(x)\rbrace$ is identified with $S^*(\pi(x)) \in (\mathcal{X}_d)^{**}=\mathcal{X}_d$. Further, we have
\begin{eqnarray}
 \Vert\lbrace h_{n,i}(x)\rbrace\Vert_{\mathcal{X}_d} = \Vert S^*(\pi(x))\Vert_{\mathcal{X}_d} \leq\parallel S\parallel\parallel x\parallel_\mathcal{X},~ x\in \mathcal{X}.
\end{eqnarray}
Letting $U={S^*}\mid_\mathcal{X}$, we have $ U(x) $ =$\lbrace h_{n,i}(x)\rbrace$ and $\Vert U\Vert\leq\Vert S\Vert$.\\
Define $R:\mathcal{X}^*\longrightarrow (\mathcal{X}_d)^*$ by $R(f)=\lbrace h(x_n)\rbrace,h\in \mathcal{X}^*$. Then
\begin{eqnarray*}
R^*(e_j)(h)=e_j(R(h)) = h(x_j), ~ \ h\in \mathcal{X}^*.
\end{eqnarray*}
 So, $R^*(e_j)=x_j, \mbox{for all} \ j \in \mathbb{N}$. Take $T=(R^*)\vert_{\mathcal{X}_d}$. Then, for $\lbrace h_{n,i} \rbrace\in \mathcal{X}_d$ we compute
\begin{eqnarray*}
T(\lbrace h_{n,i} \rbrace)=T(h_{n,i}e_i)
=\lim\limits_{n \to \infty}\sum\limits_{i=1}^{m_n}h_{n,i}T(e_i)
=\lim\limits_{n \to \infty}\sum\limits_{i=1}^{m_n}h_{n,i}x_i.
\end{eqnarray*}
 Thus, $TU(x)=\lim\limits_{n \to \infty}\sum\limits_{i=1}^{m_n}h_{n,i}(x)x_i$, for all $x\in \mathcal{X}$ and this gives $TU = L$ on $\mathcal{X}.$ Therefore, $\dfrac{1}{\Vert T\Vert}\Vert L(x)\Vert_{\mathcal{X}}\leq\Vert\lbrace h_{n,i}(x)\rbrace\Vert_{\mathcal{X}_d}.$ Then
 \begin{eqnarray*}
 \dfrac{1}{\Vert T\Vert}\Vert L(x)\Vert_{\mathcal{X}}\leq\Vert\lbrace h_{n,i}(x)\rbrace\Vert_{\mathcal{X}_d}\leq\Vert S\Vert\Vert x\Vert_{\mathcal{X}}.
 \end{eqnarray*}
 Hence, $(\{x_n\},\{h_{n,i} \}\underset{n \in \mathbb{N}}{_{i=1,2,3,...,m_n}})$ is an approximative $L$-atomic decomposition for $\mathcal{X}$ with respect to $\mathcal{X}_d.$
\end{proof}
\section{Possible Application}
One of the most important device in modern world is digital camera. In our notation a digital picture is a two-dimensional sequence, $\{h_{nm}\}.$ So, it can be seen either as an infinite length sequence with a finite number of non-zeros samples; that is
$\{h_{nm}\},~~ n, m \in \mathbb{Z},$ or as a sequence with domain $n \in \{0,1,2,..., N-1\},$ $m \in \{0,1,2,..., M-1\},$ can be expressed as a matrix:
 $$h= \begin{bmatrix} h_{0,0} & h_{0,1},&.&.&.,&h_{M-1}\\
 h_{1,0} & h_{1,1},&.&.&.,&h_{M-1}\\
  h_{N-1,0} & h_{N-1,1},&.&.&.,&h_{N-1, M-1};
 \end{bmatrix}$$ where each elements $h_{nm}$ is called a pixel and the image has $NM$ pixels. In real life for $h_{n,m}$ to represent colour image, it must have more than one component, usually, red, green and blue components are used(RGB colour space).

\end{document}